\documentclass[11pt,oneside]{article}
\usepackage{amssymb,amsmath,latexsym,amsfonts,amscd,amsthm,multirow,ctable,colortbl,mathdots,caption,array}
\usepackage{fancyhdr,tabularx,cite,mathrsfs,graphicx}
\usepackage[headings]{fullpage}
\usepackage{authblk}
\usepackage{stmaryrd}
\usepackage[all]{xy}
\usepackage{rotating}

\usepackage{setspace}

\usepackage{hyperref}
\usepackage{titlesec}

\newcommand{\bea}{\begin{eqnarray}} 
\newcommand{\eea}{\end{eqnarray}} 
\newcommand{\bee}{\begin{eqnarray*}} 
\newcommand{\eee}{\end{eqnarray*}} 
\newcommand{\al}{\begin{align*}} 
\newcommand{\eal}{\end{align*}} 
\newcommand{\be}{\begin{equation}} 
\newcommand{\ee}{\end{equation}} 
 
\newcommand{\bem}{\begin{pmatrix}} 
\newcommand{\eem}{\end{pmatrix}} 
\newcommand{\E}{\mathcal{E}}
\newcommand{\W}{\mathcal{W}}

\newcommand{\mst}{{\rm mst}}

\newcommand{\Av}{\operatorname{Av}}

\newcommand{\sk}{{\rm sk}}
\newcommand{\Ex}{\operatorname{Ex}}

\def\a{\alpha} 
\def\b{\beta} 
\def\c{\gamma} 
\def\d{\delta}

\def\f{\phi}

\def\inf{\infty}

\def\l{\lambda} 
\def\m{\mu}

\def\p{\pi}    
\def\pa{\partial}

\def\t{\tau} 
\def\th{\theta}

\def\O{\Omega}

\newcolumntype{R}{ >{$}r <{$}}
\newcolumntype{C}{ >{$}c <{$}}
\newcolumntype{L}{ >{$}l <{$}}
\newcolumntype{F}{>{\centering\arraybackslash}m{1.5cm}}

\newcommand{\gt}[1]{\mathfrak{#1}}

\newcommand{\comment}[1]{}

\newcommand{\RR}{{\mathbb R}}
\newcommand{\CC}{{\mathbb C}}
\newcommand{\ZZ}{{\mathbb Z}}
\newcommand{\N}{{\mathbb N}}
\newcommand{\QQ}{{\mathbb Q}}
\newcommand{\HH}{{\mathbb H}}

\newcommand{\ii}{{\bf i}}

\newcommand{\lab}{{\langle}}    
\newcommand{\rab}{{\rangle}}    
\newcommand{\llab}{\langle\!\langle}

\newcommand{\vv}{\mathbf v}

\newcommand{\Aut}{\operatorname{Aut}}
\newcommand{\Inn}{\operatorname{Inn}}

\newcommand{\tw}{\mathrm{tw}}
\newcommand{\tr}{\operatorname{tr}}

\newcommand{\Res}{\operatorname{Res}}

\newcommand{\ex}{\operatorname{e}} 

\newcommand{\xmod}{{\rm \;mod\;}}

\newcommand{\SL}{\operatorname{\textsl{SL}}}      
\newcommand{\GL}{{\textsl{GL}}}      

\newcommand{\g}{\gamma}	

\newcommand{\uA}{\mathrel{\raisebox{\depth}{\rotatebox{180}{$\!\!A$}}}}
\newcommand{\uvarrho}{\mathrel{\raisebox{\depth}{\rotatebox{180}{$\!\!\varrho$}}}}

\newcommand{\ulambda}{\mathrel{\raisebox{\depth}{\rotatebox{180}{$\!\lambda$}}}}

\newtheorem{thm}{Theorem}[section]

\newtheorem{lem}[thm]{Lemma}
\newtheorem{prop}[thm]{Proposition}

\theoremstyle{definition}

\theoremstyle{remark}

\numberwithin{equation}{section}

\begin{document}

\setstretch{1.4}

\title{
\vspace{-35pt}
    \textsc{\huge{{M}eromorphic {J}acobi {F}orms of {H}alf-{I}ntegral {I}ndex and {U}mbral {M}oonshine {M}odules}
    }
    }

\renewcommand{\thefootnote}{\fnsymbol{footnote}} 


\renewcommand{\thefootnote}{\arabic{footnote}} 

\author[1]{Miranda C. N. Cheng\thanks{mcheng@uva.nl}}
\author[2]{John F. R. Duncan\thanks{john.duncan@emory.edu}}

\affil[1]{Institute of Physics and Korteweg-de Vries Institute for Mathematics\\
University of Amsterdam, Amsterdam, the Netherlands\footnote{On leave from CNRS, France.}.}
\affil[2]{Department of Mathematics and Computer Science\\
Emory University, Atlanta, GA 30322, U.S.A.}

\date{} 

\maketitle

\abstract{
In this work we consider an association of meromorphic Jacobi forms of half-integral index to the pure D-type cases of umbral moonshine, and solve the 
module problem for four of these cases by constructing vertex operator superalgebras that realise the corresponding meromorphic Jacobi forms as graded traces.
We also present a general discussion of meromorphic Jacobi forms with half-integral index and their relationship to mock modular forms.
}

\clearpage

\tableofcontents

\section{Introduction}\label{sec:intro}

Moonshine connects modular objects and finite groups in an interesting way. 
Starting from 2010, the discovery of Mathieu moonshine for $M_{24}$ \cite{Eguchi2010} has initiated a new wave of activity in the study of moonshine. 
One development arising is the realisation that this $M_{24}$ moonshine is but one of 23 instances of the so-called umbral moonshine \cite{UM,MUM}. 
Umbral moonshine associates 
a vector-valued mock modular form $H^{(\ell)}_g=(H^{(\ell)}_{g,r})$, with a certain level and multiplier system, to each conjugacy class $[g]$ of a finite group $G^{(\ell)}$, where $\ell$ is a symbol (called lambency) which indexes the 23 Niemeier lattices (and also indexes certain genus zero groups of isometries of the upper half-plane, cf. \cite{omjt}). The group $G^{(\ell)}$ can be defined explicitly as the group of outer automorphisms of the Niemeier lattice corresponding to $\ell$ (cf. \S2.4 of \cite{MUM}).
One of the umbral moonshine conjectures, proven for Mathieu moonshine in \cite{MR3539377} and proven for umbral moonshine in general in \cite{umrec}, 
then states that the $H^{(\ell)}_g$ can be identified with the graded characters of a certain infinite-dimensional $G^{(\ell)}$-module.

This is reminiscent of monstrous moonshine \cite{MR554399}, which attaches to each conjugacy class $[g]$ of the monster group a  certain Hauptmodul $T_g(\tau)$ for a certain subgroup $\Gamma_g$ of $\SL_2(\RR)$ that defines a genus zero quotient of the upper half-plane. These functions $T_g(\tau)$ were ultimately shown by Borcherds \cite{MR1172696} to coincide with the graded characters of a certain infinite-dimensional module $V^\natural$ for the monster group, which was constructed earlier by Frenkel--Lepowsky--Meurman \cite{FLMBerk,FLMPNAS} and which famously possesses vertex operator algebra structure \cite{Bor_PNAS,FLM}.

We would like to have a similar understanding of the modules for umbral moonshine. Arguably, this is the most important outstanding question in the study of umbral moonshine. The first case of umbral moonshine for which the module problem was solved is the case where $\ell=30+6,10,15$, the Niemeier lattice is $E_8^{\oplus 3}$, and $G^{(\ell)}\cong S_3$ \cite{MR3649360}. This is the unique case for which the Niemeier lattice coincides with its root lattice. 
So far the techniques of \cite{MR3649360} have not been successfully extended to other cases.

In the more recent work \cite{umvan4}, and in the present paper, the focus is shifted from 
the vector-valued mock modular forms $H^{(\ell)}_g$
to certain meromorphic Jacobi forms. 
The former may be obtained from the latter 
once a canonically defined polar part that captures the poles of the latter 
is removed. We explain this in detail in \S\ref{sec:mjf}. The meromorphic Jacobi forms 
involved in the cases of umbral moonshine discussed here and in \cite{umvan4} admit expressions as sums of ratios of infinite products, and this makes it possible to realise them as (twined) partition functions of certain free chiral conformal field theories. Using this strategy, vertex operator superalgebras are constructed that possess automorphisms containing the corresponding umbral groups $G^{(\ell)}$, and are such that the twined partition functions arising coincide with the  meromorphic Jacobi forms 
specified by umbral moonshine. In this way, the umbral moonshine modules for two of the pure A-type cases, corresponding to $\ell=7$ and $\ell=13$, have been constructed in \cite{umvan4} (while partial solutions to the module problem are given in loc. cit. for $\ell=4$ and $\ell=5$). The corresponding Niemeier lattices are those with the root systems $A_6^{\oplus 4}$ and $A_{12}^{\oplus 2}$, respectively.

The main objective of this work is the construction of umbral moonshine modules for four of the five pure D-type umbral moonshine cases, corresponding to $\ell=10+5$, $\ell=14+7$, $\ell=22+11$ and $\ell=46+23$, for which the corresponding Niemeier lattices are those whose roots systems are $D_6^{\oplus 4}$, $D_8^{\oplus 3}$, $D_{12}^{\oplus 2}$ and $D_{24}$, respectively. 
We achieve this in \S\ref{sec:ummod}. See Theorems \ref{thm:ummod:10+5-psig}, \ref{thm:ummod:7/2-psig}, \ref{thm:ummod:11/2-psig} and \ref{thm:ummod:23/2-psig}. Our methods are similar to those of \cite{umvan4}, but diverge from loc. cit. in an apparently powerful aspect, by considering meromorphic Jacobi forms of half-integral index, rather than those of integral index which are attached to the pure D-type cases of umbral moonshine in \cite{MUM}. 

In order to support our approach we present results on meromorphic Jacobi forms of half-integral index and their relationship to mock modular forms in \S\ref{sec:mjf}. Theorem \ref{thm:mjf:polarfinitedec} represents an extension to half-integral index of the canonical splitting of a meromorphic Jacobi form into polar and finite parts, which was first given for integral index in \cite{Dabholkar:2012nd}. Theorem \ref{non-thm} precisely formulates the relationship between mock modular forms and meromorphic Jacobi forms of weight $1$ and half-integral index that have simple poles only at torsion points, and satisfy a certain growth condition.

The paper is structured in the following way. In \S\ref{sec:mjf} we discuss meromorphic Jacobi forms of half-integral index and their relationship to mock modular forms. We review elliptic forms, theta series and Eicher--Zagier operators of half-integral index in \S\ref{sec:mjf:ell}. We review holomorphic, skew-holomorphic and meromorphic Jacobi forms in \S\ref{sec:mjf:mjf}, and also review there the construction of the finite part of a meromorphic Jacobi form due to Dabholkar--Murthy--Zagier \cite{Dabholkar:2012nd}. In \S\ref{sec:mjf:polar} we review the construction of the polar part of a meromorphic Jacobi form from its poles, assuming that they are simple and constrained to lie at torsion points. This extends the discussion of \S8.2 of \cite{Dabholkar:2012nd} to half-integral indexes. In \S\ref{sec:mjf:mock} we specialise to meromorphic Jacobi forms of weight $1$ satisfying a growth condition, and prove our main result (see Theorem \ref{non-thm}) on the mock modularity of the finite parts and theta coefficients of such forms. In \S\ref{sec:mjf:umbral} we explain the construction which converts the vector-valued mock modular forms of pure D-type umbral moonshine into meromorphic Jacobi forms of half-integral index.

Our main constructions appear in \S\ref{sec:ummod}, where we obtain vertex algebraic realisations of the half-integral index meromorphic Jacobi forms presented in \S\ref{sec:mjf:umbral}, for four of the pure D-type cases of umbral moonshine. 
The case of $D_6^{\oplus 4}$ is treated in \S\ref{sec:ummod:10+5} (see Theorem \ref{thm:ummod:10+5-psig}), the case of $D_8^{\oplus 3}$ is treated in \S\ref{sec:ummod:14+7} (see Theorem \ref{thm:ummod:7/2-psig}), the case of $D_{12}^{\oplus 2}$ is treated in \S\ref{sec:ummod:22+11} (see Theorem \ref{thm:ummod:11/2-psig}), and $D_{24}$ is treated in \S\ref{sec:ummod:46+23} (see Theorem \ref{thm:ummod:23/2-psig}).
We conclude in \S\ref{sec:disc} with discussions on the relation between the known modules for pure A-type and pure D-type umbral moonshine, as well as the possible physical relevance of our constructions.

\section{Meromorphic and Mock Jacobi Forms}\label{sec:mjf}

In this section we describe a relationship between meromorphic Jacobi forms and mock modular forms. This relationship goes back to work of Zwegers \cite{zwegers}, which was subsequently developed by Dabholkar--Murthy--Zagier in \cite{Dabholkar:2012nd}. Here we follow \cite{Dabholkar:2012nd} quite closely, but extend the discussion to include half-integral index, and refine it for meromorphic Jacobi forms with weight $1$.

\subsection{Elliptic Forms}\label{sec:mjf:ell}

We first discuss elliptic forms with half-integral index.
For this set $\ex(x):=e^{2\pi i x}$ and let $\HH:=\{\tau\in\CC\mid \Im(\tau)>0\}$ denote the upper half-plane.
For $m \in {1\over 2}\ZZ$ define the index $m$ {\em elliptic action} of 
$\ZZ^2$ on functions $\phi:\HH\times\CC\to \CC$ by setting
\begin{equation} \label{elliptic}
(\phi\lvert_{m} (\lambda,\mu) )(\tau,z) := e\left( m\lambda^2 \tau + 2m\lambda z + {(\lambda +\mu) m}\right) \, \phi(\tau, z+\lambda \tau +\mu)
\end{equation}
for $(\lambda,\mu)\in \ZZ^2$. With this definition we call a smooth function $\phi:\HH\times \CC\to \CC$ an {\em elliptic form}\footnote{Our nomenclature follows \cite{Dabholkar:2012nd}, except that elliptic forms are assumed to be holomorphic in that work, and the focus there is on integral index.} 
with index $m$ if $(\phi\lvert_{m} (\lambda,\mu) ) = \phi$ for all $(\lambda,\mu)\in \ZZ^2$. 
Write $\E_m$ for the space of elliptic forms of index $m\in\frac12\ZZ$.

For $m\in{1\over 2}\ZZ$ and $r \in \ZZ+m$ define the {\em index $m$ theta function} $\theta_{m,r}:\HH\times \CC\to \CC$ by setting 
\begin{gather}\label{eqn:thetamr}
\theta_{m,r}(\tau,z):=\sum_{\substack{\ell\in\ZZ+m\\\ell=r\xmod 2m}}\ex(m\ell)y^{\ell}q^{\frac{\ell^2}{4m}}
\end{gather}
where $q:=\ex(\tau)$ and $y:=\ex(z)$. 
Also 
define $ \theta^1_{m,r}(\t) := \frac{1}{2\p i}{\pa \over \pa z} \theta_{m,r}(\t,z) \lvert_{z=0}$. 
Then a smooth function $\phi:\HH\times \CC\to \CC$ belongs to $\E_m$ if and only if it can be written in the form 
\begin{gather}\label{eqn:thetadec}
\phi(\tau,z)=\sum_{r\xmod 2m}h_r(\tau)\theta_{m,r}(\tau,z)
\end{gather}
for some $2m$ {\em theta coefficient} functions $h_r:\HH\to \CC$. We emphasise here that $m\in\frac12\ZZ$, and the summation in (\ref{eqn:thetadec}) is over representatives for $\ZZ+m$ modulo translation by $2m$. 

For an explicit example note that the $m=r=\frac12$ case of (\ref{eqn:thetamr}) recovers a classical Jacobi theta function
\begin{gather}\label{eqn:Jacobithetaone}
\begin{split}
	\theta_{\frac12,\frac12}(\tau,z)
	&=\sum_{n\in\ZZ}i(-1)^{n+1}y^{n+\frac12}q^{\frac12(n+\frac12)^2}\\
	&=-iy^{\frac12}q^{\frac18}\prod_{n>0}(1-y^{-1}q^{n-1})(1-yq^n)(1-q^n).
\end{split}
\end{gather}
We will denote $\theta_{\frac12,\frac12}(\tau,z)$ by $\theta_1(\tau,z)$ in \S\S\ref{sec:ummod},\ref{sec:disc} (cf. (\ref{eqn:ummod:10+5-jactheta})).

Elliptic forms come equipped with some naturally defined symmetry. To explain this 
define a function 
$m\mapsto \tilde{m}$ on $\frac12\ZZ$ by setting
\begin{gather}\label{eqn:tildem}
	\tilde{m}:=\begin{cases} m&\text{ if $m\in\ZZ$,}\\ 2m&\text{ if $m\in\ZZ+\frac12$.}\end{cases}
\end{gather}
Then the group $O_{\tilde{m}}:=\{a\in \ZZ/2\tilde{m}\ZZ\mid a^2=1\xmod 4\tilde{m}\}$ acts naturally on ${\E}_m$ via
\begin{gather}
	\phi\cdot a := \sum_{r\xmod 2m} h_r\theta_{m,ra}
\end{gather}
when $\phi=\sum_{r\xmod 2m}h_r\theta_{m,r}$ is the theta decomposition (\ref{eqn:thetadec}) of $\phi\in {\E}_m$. 

It develops that $O_{\tilde{m}}$ is naturally isomorphic to the group $\Ex(\tilde{m})$ of exact divisors of $\tilde{m}$. We obtain an explicit isomorphism $O_{\tilde m}\to \Ex(\tilde m)$, to be denoted $n\mapsto a(n)$, by letting $a(n)$ be the unique element of $\ZZ/2\tilde{m}\ZZ$ such that $a(n)=1\xmod 2\frac{\tilde{m}}{n}$ and $a(n)=-1\xmod 2n$.
With this definition we have 
\begin{gather}
\phi|\W_m(n)=\phi\cdot a(n)
\end{gather}
for $\phi\in \E_m$ and $n\in \Ex(\tilde{m})$, where 
\begin{gather}\label{eqn:Wmn}
	(\phi|\W_m(n))(\tau,z):=\tfrac1n\sum_{a,b=0}^{n-1}\ex(m(\tfrac{a^2}{n^2}\tau+2\tfrac{a}{n}z+\tfrac{ab}{n^2}+ab+a+b))\phi(\tau,z+\tfrac{a}{n}\tau+\tfrac{b}{n}).
\end{gather}

We call $\W_m(n)$ an {\em Eichler--Zagier operator} of index $m$. The $\W_m(n)$ for integral index first appeared in \cite{eichler_zagier}. Note that (\ref{eqn:Wmn}) is well-defined and stabilises $\E_m$ for $n$ an arbitrary divisor of $\tilde{m}$. 

\subsection{Meromorphic Jacobi Forms}\label{sec:mjf:mjf}

We now explain precisely what we mean by a meromorphic Jacobi form. 
Recall that a {\em holomorphic Jacobi form} of weight $k\in\ZZ$ and index $m\in\frac12\ZZ$ for a group $\Gamma<\SL_2(\ZZ)$ is an elliptic form $\phi=\sum_{r\xmod 2m} h_r\theta_{m,r}$ of index $m$ 
that is invariant for the weight $k$ index $m$ {\em modular action} 
\begin{gather}\label{eqn:modularweightkaction}
\left(\f|_{k,m}\left(\begin{smallmatrix}a&b\\c&d\end{smallmatrix}\right)\right)(\tau,z):=
\f\left(\tfrac{a\t+b}{c\t+d},\tfrac{z}{c\t+d}\right)
e\left(- m \tfrac{cz^2}{c\tau+d}\right)
(c\tau+d)^{-k}
\end{gather}
of $\left(\begin{smallmatrix}a&b\\c&d\end{smallmatrix}\right)\in \Gamma$, and whose theta coefficients (\ref{eqn:thetadec}) are holomorphic and bounded near any cusp of $\Gamma$.
A {\em skew-holomorphic Jacobi form} is defined similarly, just replacing (\ref{eqn:modularweightkaction}) with the weight $k$ index $m$ {\em skew-modular action}
\begin{gather}\label{eqn:skewmodularweightkaction}
\left(\f|_{k,m}^\sk\left(\begin{smallmatrix}a&b\\c&d\end{smallmatrix}\right)\right)(\tau,z):=
\f\left(\tfrac{a\t+b}{c\t+d},\tfrac{z}{c\t+d}\right)
e\left(- m \tfrac{cz^2}{c\tau+d}\right)
(c\tau+d)^{-k+1}|c\tau+d|^{-1}
\end{gather}
of $\left(\begin{smallmatrix}a&b\\c&d\end{smallmatrix}\right)\in \Gamma$, and requiring the theta coefficients (\ref{eqn:thetadec}) to be anti-holomorphic and bounded near any cusp of $\Gamma$.
For us a {\em meromorphic Jacobi form} 
is a quotient $\psi=\frac{\phi_1}{\phi_2}$ of holomorphic Jacobi forms 
such that for any fixed $\tau\in \HH$ the function 
$z\mapsto\phi_2(\tau,z)$ is not identically zero. With this definition $z\mapsto\psi(\tau,z)$ 
is meromorphic on $\CC$ for all $\tau\in \HH$.
If $k_i$ and $m_i$ are the weight and index (respectively) of $\phi_i$ then $k=k_1-k_2$ and $m=m_1-m_2$ are the weight and index (respectively) of $\psi$. 

A {\em weak holomorphic Jacobi form} of weight $k$ and index $m$ is a holomorphic function $\phi\in \E_m$ that is invariant for the weight $k$ modular action (\ref{eqn:modularweightkaction}), and is such that $\tau\mapsto(\phi|_{k,m}\gamma)(\tau,z)$ is bounded as $\Im(\tau)\to \infty$ for any fixed $z\in \CC$ and $\gamma\in \SL_2(\ZZ)$.
We will say that a meromorphic Jacobi form $\psi$ of weight $k$ and index $m$ is a {\em weak meromorphic Jacobi form} if, for fixed $z\in\CC$ and arbitrary $\gamma\in \SL_2(\ZZ)$, the function 
$\tau\mapsto(\psi|_{k,m}\gamma)(\tau,z)$ is bounded as $\Im(\tau)\to \infty$ along any path that avoids poles. In general we have that $\tau\mapsto (\psi|_{k,m}\gamma)(\tau,z)$ is $O(e^{C\Im(\tau)})$ as $\Im(\tau)\to \infty$ (along paths that avoid poles), for some $C>0$ that is independent of $\gamma$ and $z$.
In the rest of this section we will focus on meromorphic Jacobi forms for $\SL_2(\ZZ)$, but meromorphic Jacobi forms for proper subgroups $\Gamma<\SL_2(\ZZ)$ will appear in \S\ref{sec:ummod}.

If $\psi$ is a meromorphic Jacobi form then from invariance under the index $m$ elliptic action (\ref{elliptic}) it follows that the function $z\mapsto \psi(\tau,z)\ex(-rz)$ is invariant under $z\mapsto z+1$ for $r\in \ZZ+m$. So for fixed $\tau\in \HH$ the integral 
\begin{gather}\label{eqn:hrint}
	h_r(\tau):=\ex(mr)q^{\frac{r^2}{4m}}\int_{\RR/\ZZ}\psi(\tau,z-\tfrac{r}{2m}\tau)\ex(-rz){\rm d}z
\end{gather}
is well-defined for $r\in \ZZ+m$, so long as $z\mapsto\psi(\tau,z)$ has no poles of the form $z=-\tfrac{r}{2m}\tau+\b$ for $\b\in \RR$. If $z=-\tfrac{r}{2m}\tau+\b$ is a pole for some $\b\in \RR$ then choose $\b_0\in \RR$ so that $z=-\tfrac{r}{2m}\tau+\b_0$ is not a pole, let $\gamma$ be a deformation of the line segment $[\b_0,\b_0+1]$ that passes just above the poles of $z\mapsto \psi(\tau,z)$, let $\bar{\gamma}$ be the image of $\gamma$ under complex conjugation, and replace $\int_{\RR/\ZZ}$ with $\frac12\left(\int_{\gamma}+\int_{\bar\gamma}\right)$ in (\ref{eqn:hrint}). For $m\in \ZZ$ this recovers the definition (8.2) in \cite{Dabholkar:2012nd}. For general $m\in \frac12\ZZ$ we check using the behavior of $\psi$ under $z\mapsto z+\tau$ (cf. (\ref{elliptic})) that $h_r(\tau)$ depends only on $r\xmod 2m$. 

If $\psi$ happens to be holomorphic then the $h_r$ of (\ref{eqn:hrint}) are the theta coefficients (\ref{eqn:thetadec}) of $\psi$. If $\psi$ has poles it cannot admit such a decomposition, but following \cite{Dabholkar:2012nd} we may define the  {\em finite part} $\psi^F$ of $\psi$ by setting
\begin{gather}\label{eqn:psiF}
	\psi^F(\tau,z):=\sum_{r\xmod 2m}h_r(\tau)\theta_{m,r}(\tau,z),
\end{gather}
where the sum is over $r\in \ZZ+m$ modulo $2m$, and the $\theta_{m,r}$ are as in (\ref{eqn:thetamr}). 
The construction (\ref{eqn:hrint}) does not preserve modular invariance (\ref{eqn:modularweightkaction}) in general.
It will develop in \S \ref{sec:mjf:mock} that 
$\psi^F$ is a mock Jacobi form, and the $h_r$ are mock modular forms, when certain hypotheses are satisfied by $\psi$ (cf. Theorem \ref{non-thm}).

\subsection{Residues}\label{sec:mjf:polar}

In order to relate meromorphic Jacobi forms to mock modular forms we follow \cite{Dabholkar:2012nd} by restricting to meromorphic Jacobi forms 
$\psi$ such that the poles of $z\mapsto\psi(\tau,z)$, for fixed $\tau\in \HH$, are constrained to lie in the set of {\em torsion points} $\QQ\tau+\QQ$. For simplicity we further restrict to meromorphic Jacobi forms 
whose poles are simple. (Single poles and double poles are discussed in \cite{Dabholkar:2012nd}.) 
So for $m\in \frac12\ZZ$ let $J^\mst_{k,m}$ denote the space of meromorphic Jacobi forms $\psi$ of weight $k$ and index $m$ for $\Gamma=\SL_2(\ZZ)$ 
such that the poles of the functions $z\mapsto \psi(\tau,z)$ are simple, and occur only at torsion points $z\in \QQ\tau+\QQ$. 

Our next objective is to present an explicit construction of the difference $\psi^P=\psi-\psi^F$, which is called the {\em polar part} of $\psi$.
To prepare for this 
observe that $\ZZ^2\rtimes \SL_2(\ZZ)$ acts naturally from the right on $\QQ^2$ according to the rules 
\begin{gather}
\label{action_Jac_group}
(\alpha,\beta)\cdot(\lambda,\mu):=(\alpha+\lambda,\beta+\mu),\quad
(\alpha,\beta)\cdot\left(\begin{smallmatrix}a&b\\c&d\end{smallmatrix}\right):=(\alpha a+\beta c,\alpha b+\beta d).
\end{gather}
Here $(\alpha,\beta)\in\QQ^2$, $(\lambda,\mu)\in\ZZ^2$ and $\left(\begin{smallmatrix}a&b\\c&d\end{smallmatrix}\right)\in\SL_2(\ZZ)$. 
It will be useful to have a description of the orbits of this action. With this in mind we
set 
\begin{gather}\label{eqn:Sn}
S_n:=\left\{(\tfrac{k}n,\tfrac{l}n)\in\QQ^2\mid k,l\in\ZZ,\, \gcd(k,l,n)=1\right\}
\end{gather}
for $n\in \ZZ^+$. 
\begin{lem}\label{lem:Sn}
The sets $S_n$ for $n\in \ZZ^+$ are the orbits for the action of $\ZZ^2\rtimes\SL_2(\ZZ)$ on $\QQ^2$.
\end{lem}
\begin{proof}
It is straightforward to check that $S_n$ is stable under the action of 
$\ZZ^2\rtimes\SL_2(\ZZ)$. 
To see that $\ZZ^2\rtimes\SL_2(\ZZ)$ acts transitively observe that if $s=(\tfrac kn,\tfrac ln)\in S_n$ where $k,l\in\ZZ$ and $\gcd(k,l,n)=1$ then there exists $j\in\ZZ$ such that $\gcd(k,l+jn)=1$. So after replacing $s$ with $s\cdot (0,j)$ we may assume that $\gcd(k,l)=1$. Then for $x,y\in \ZZ$ such that $kx+ly=1$ we have $s\cdot\left(\begin{smallmatrix}x&-l\\y&k\end{smallmatrix}\right)=(\frac 1n,0)$. So $S_n$ is just the orbit of $(\frac 1n,0)$ under $\ZZ^2\rtimes\SL_2(\ZZ)$.

To see that the $S_n$ exhaust $\QQ^2$ let $S\subset \QQ^2$ be an arbitrary orbit for $\ZZ^2\rtimes\SL_2(\ZZ)$. If $(\frac kn,\frac ln)\in S$ where $k,l,n\in\ZZ$ (and the components are not necessarily fractions in lowest form), then $n\alpha$ and $n\beta$ are integers for every $(\a,\b)\in S$. So there is a smallest positive integer $n$ such that $n\alpha$ and $n\beta$ are integers for every $(\a,\b)\in S$. For this $n$ there must be $(\frac kn,\frac ln)\in S$ such that $\gcd(k,l,n)=1$. Then $S=S_n$ by the above.
\end{proof}

Given $\psi\in J^\mst_{k,m}$ 
define $S(\psi)\subset \QQ^2$ to be the set of pairs $s=(\alpha,\beta)$ such that $z\mapsto \psi(\tau,z)$ has a pole 
at $z_s(\t):=\alpha\tau+\beta$ for some $\tau\in\HH$. When there is no risk of confusion we write $z_s$ in place of $z_s(\tau)$. 
For $s=(\a,\b)\in S(\psi)$ we follow \cite{Dabholkar:2012nd} in defining
\begin{gather}\label{eqn:Dstau}
D_s(\t)  := 2\p i \ex(m\a z_s) \Res_{z=z_s}\psi(\t,z).
\end{gather}
From the invariance of $\psi$ under the elliptic (\ref{elliptic}) and modular (\ref{eqn:modularweightkaction}) actions
we deduce that
\begin{align}\label{ellip_trans_S}
D_{s\cdot(\l,\mu)}(\t)  
& = \ex(m(\a\mu-\b\l+\l\m+\l+\m)) D_{s}(\t),
\\ 
\label{modular_trans_S}
D_{s\cdot\g}(\t)
&= D_{s}(\g\t) 
(c\tau+d)^{1-k}, 
\end{align}
for $(\lambda,\mu)\in\ZZ^2$ and $\g=\left(\begin{smallmatrix}*&*\\c&d\end{smallmatrix}\right)\in\SL_2(\ZZ)$. 
(See Proposition 8.3 of \cite{Dabholkar:2012nd} for details on (\ref{modular_trans_S}).) 
So in particular, $S(\psi)$ is 
stable for the action of $\ZZ^2\rtimes\SL_2(\ZZ)$,
and the $D_s(\tau)$ are modular forms of weight $k-1$. From Lemma \ref{lem:Sn} we deduce that $S(\psi)$ is a union of finitely many of the $S_n$. We formulate the modularity of the $D_s(\tau)$ more precisely as follows.

\begin{prop}\label{prop:Dsmod}
Let $\psi\in J_{k,m}^\mst$ for $k\in \ZZ$ and $m\in \frac12\ZZ$. If $s\in S(\psi)\cap S_n$ then $D_s(\tau)$ is a weakly holomorphic modular form of weight $k-1$ for $\Gamma(2n^2)$. If $\psi$ is weak 
then $D_s(\tau)$ is a holomorphic modular form of weight $k-1$ for $\Gamma(2n^2)$.
\end{prop}
\begin{proof}
The growth conditions on $D_s(\tau)$ near cusps follow from the definitions of meromorphic and weak meromorphic Jacobi form (cf. \S\ref{sec:mjf:mjf}). To check modularity write $s=(\frac kn,\frac ln)$ where $\gcd(k,l,n)=1$ and 
let $\gamma=\left(\begin{smallmatrix}1+2n^2a'&2n^2b'\\2n^2c'&1+2n^2d'\end{smallmatrix}\right)\in\Gamma(2n^2)$. 
Then
\begin{gather}
s\cdot \gamma=(\tfrac kn+2n\lambda,\tfrac ln+2n\mu)=s\cdot (2n\lambda,2n\mu)
\end{gather}
where $\lambda=ka'+lc'$ and $\mu=kb'+ld'$ are integers. Now $D_{s\cdot(2n\lambda,2n\mu)}(\tau)=D_s(\tau)$ according to (\ref{ellip_trans_S}), so from (\ref{modular_trans_S}) we have
\begin{gather}
	D_s(\gamma\tau)(c\tau+d)^{1-k}=D_{s\cdot\gamma}(\tau)=D_{s\cdot(2n\lambda,2n\mu)}(\tau)=D_s(\tau)
\end{gather}
where $c=2n^2c'$ and $d=1+2n^2d'$. This proves the claim.
\end{proof}
The next result will be applied in \S\ref{sec:mjf:mock}. 
\begin{lem}\label{lem:DsFou}
Let $\psi\in J_{k,m}^\mst$ for $m\in \frac12\ZZ$ and suppose that $ S_n\subset S(\psi)$. Then for $s=(\frac 1n,0)$ the Fourier expansion of $D_{s}(\tau)$ takes the form
\begin{gather}\label{eqn:DsFou}
	D_s(\tau) = \sum_{\substack{N\in \ZZ\\N=2m(n+1)\xmod 2n}} c_s(N)q^{\frac{N}{2n^2}}.
\end{gather}
In particular, the constant term of $D_s(\tau)$ vanishes unless $n$ divides $\tilde{m}$.
\end{lem}
\begin{proof}
By Proposition \ref{prop:Dsmod} we know that $D_s(\tau)$ has a Fourier expansion of the form $D_s(\tau)=\sum_{N\in\ZZ}c_s(N)q^{\frac{N}{2n^2}}$. We also have $s\cdot(0,1) = s\cdot \gamma$ for $\gamma=\left(\begin{smallmatrix}1&n\\0&1\end{smallmatrix}\right)$. So applying (\ref{ellip_trans_S}) and (\ref{modular_trans_S}) we obtain
\begin{gather}
	D_s(\tau+n) = D_{s\cdot\gamma}(\tau)= D_{s\cdot (0,1)}(\tau) = \ex (m(\tfrac1n+1))D_s(\tau).
\end{gather}
This shows that $c_s(N)=0$ unless $\frac{N}{2n}=m\frac{n+1}{n}\xmod 1$. This is the same as $N=2m(n+1)\xmod 2n$ so (\ref{eqn:DsFou}) holds. Of course $2m(n+1)=2m\xmod 2n$ if $m\in\ZZ$ so $c_s(0)$ must vanish unless $n$ divides $m=\tilde{m}$ when $m\in \ZZ$. If $m\in \ZZ+\frac12$ then $\tilde{m}=2m$ so $c_s(0)=0$ unless $\tilde{m}(n+1)$ is divisible by $2n$. This forces $n$ to be a divisor of $\tilde{m}$. The proof is complete.
\end{proof}

Roughly speaking, we will obtain $\psi^P$ by averaging the residues of $z\mapsto \psi(\tau,z)$ that appear in some fundamental parallelogram for the action of $\ZZ\tau+\ZZ$ on $\CC$.
In order to realise this we define the index $m$ {\em averaging operator}
\begin{gather}\label{eqn:AvmFy}
\Av_m(F(y)) := \sum_{k\in \ZZ} (-1)^{2mk} q^{mk^2}y^{2mk} F(q^k y).
\end{gather} 
For the purpose at hand, $F(y)$ will be a rational function in $y$ when $m\in \ZZ$, and a rational function in $y^{\frac12}$ when $m\in \ZZ+\frac12$. 
Indeed, for $m\in\frac12\ZZ$ and $c\in \QQ$ we set
\begin{gather}\label{eqn:Rmc}
{\cal R}_{m,c}(y) := 
\begin{cases}
y^c\dfrac{1}{2}\dfrac{y+1}{y-1}   & \text{ if $m-c \in \ZZ$,}\\
y^{\lceil c\rceil}\dfrac{1}{y-1}   & \text{ if $m \in \ZZ$, $c\notin\ZZ$,}\\
y^{\lceil c-\frac12\rceil}\dfrac{y^{\frac12}}{y-1}   & \text{ if $m \in \ZZ+\frac12$, $c\notin\ZZ+\frac12$.} 
\end{cases} 
\end{gather}
We then define the {\em universal Appell--Lerch sum} 
${\cal A}^s_m(\t,z)$ for $m\in\frac12\ZZ$ and $s=(\a,\b)\in \QQ^2$ by setting
\begin{gather}\label{eqn:Asm}
{\cal A}^{s}_m(\t,z):= \ex(-m\a z_s) \Av_m({\cal R}_{m,-2m\a}(yy_s^{-1}))
\end{gather}
where $y_s:=\ex(z_s)=\ex(\beta)q^\a$. Note that (\ref{eqn:Asm}) reduces to (8.17) of \cite{Dabholkar:2012nd} when $m\in \ZZ$. Also, the function ${\cal R}_{m,c}$ admits a more uniform description as a series:
\begin{gather}
	{\cal R}_{m,c}(y) = 
	\begin{cases}
		-\sum_{\ell\geq c} y^{\ell}(1-\frac12\delta_{\ell,c})&\text{ for $|y|<1$,}\\
		\quad\sum_{\ell\leq c} y^{\ell}(1-\frac12\delta_{\ell,c})&\text{ for $|y|>1$,}
	\end{cases}
\end{gather}
where the summations are over $\ell\in \ZZ+m$. 

Using the identity ${\cal R}_{m,c+k}(y)={\cal R}_{m,c}(y)y^k$ for $k\in \ZZ$ we obtain 
\begin{gather}\label{eqn:Amstrans}
	{\cal A}_m^{s\cdot(\lambda,\mu)}(\tau,z)=\ex(-m(\alpha\mu-\beta\lambda+\lambda\mu+\lambda+mu)){\cal A}_m^s(\tau,z)
\end{gather}
for $s=(\a,\b)\in \QQ^2$ and $(\lambda,\mu)\in \ZZ^2$. So comparing (\ref{ellip_trans_S}) with (\ref{eqn:Amstrans}) we see that it makes sense to define
\begin{gather}\label{eqn:psiP}
	\psi^P(\tau,z):=\sum_{s\in S(\psi)/\ZZ^2}D_s(\tau){\cal A}_m^s(\tau,z)
\end{gather}
where the sum is over any set of representatives for the action of $\ZZ^2$ on $S(\psi)$.

The next result recovers Theorem 8.1 of \cite{Dabholkar:2012nd} when $m\in \ZZ$. Given the definitions we have made, the proof for $m\in \ZZ+\frac12$ is essentially the same.
\begin{thm}\label{thm:mjf:polarfinitedec}
Let $\psi\in J^\mst_{k,m}$ for $k\in \ZZ$ and $m\in \frac12\ZZ$. Then we have
\begin{gather}
	\psi(\tau,z)=\psi^F(\tau,z)+\psi^P(\tau,z),
\end{gather} 
where $\psi^F$ is defined by (\ref{eqn:psiF}) and $\psi^P$ is defined by (\ref{eqn:psiP}).
\end{thm}

\subsection{Mock Modularity}\label{sec:mjf:mock}

It is shown in \S8.3 of \cite{Dabholkar:2012nd} that if $m\in \ZZ$ then the theta coefficients $h_r$ (cf. (\ref{eqn:hrint})) of the finite part $\psi^F$ (cf. (\ref{eqn:psiF})) of a meromorphic Jacobi form $\psi\in J_{k,m}^\mst$ are (mixed) mock modular forms, whose shadows admit concrete expressions in terms of the residue forms $D_s$ (cf. (\ref{eqn:Dstau})) and the theta functions $\theta_{m,r}^1$ (cf. (\ref{eqn:thetamr})).
Our focus in this work is on weak meromorphic Jacobi forms of weight $1$ (with half-integral index), so in this section we discuss mock modularity in that specific setting. We begin by showing that $S(\psi)$ and the $D_s$ are greatly constrained by these hypotheses.
\begin{prop}\label{prop:Dswt1}
If $\psi$ is a weak meromorphic Jacobi form in $J^\mst_{1,m}$ for $m\in \frac12\ZZ$ then $D_s(\tau)$ is a constant function for all $s\in S(\psi)$, and if $S_n\subset S(\psi)$ then $n$ divides $\tilde{m}$.
\end{prop}
\begin{proof}
Let $s\in S(\psi)$. Then $s\in S_n$ for some $n$, and Proposition \ref{prop:Dsmod} shows that $D_s(\tau)$ is a bounded weakly holomorphic modular form of weight $0$ for $\Gamma(2n^2)$. We conclude that $D_s(\tau)=D_s$ is constant. 
Note that $D_s\neq 0$ for otherwise $s$ would not be in $S(\psi)$. From (\ref{ellip_trans_S}) and (\ref{modular_trans_S}) and the fact that $\ZZ^2\rtimes\SL_2(\ZZ)$ acts transitively on $S_n$ (Lemma \ref{lem:Sn}) we have that $D_s$ is a non-zero constant times $D_{(\frac 1n,0)}$. Applying Lemma \ref{lem:DsFou} we 
obtain that $n$ divides $\tilde{m}$ (recall (\ref{eqn:tildem})). The proof is complete.
\end{proof}

We now aim to show that if $m\in\frac12\ZZ$ and $\psi\in J^\mst_{1,m}$ is weak then the finite part $\psi^F$ is a mock Jacobi form of weight $1$, and its theta coefficients $h_r$ constitute a vector-valued mock modular form $h=(h_r)$ of weight $\frac12$.
To identify their shadows let us focus momentarily on the poles lying on lattice points $\ZZ\t+\ZZ$.
For half-integral $m$, the mock modular property of the polar part (and hence also the finite part) 
can be analyzed using properties of the higher level Appell functions introduced by Zwegers (see Definition 1.3. of \cite{ZWEGERS2010MultivariableAF}). 
Define 
\be
\mu_{{m},0}(\t,z):={\cal A}_m^{(0,0)}(\t,z),
\ee
and define its {\em completion} by setting 
\be\label{eqn:mucomp}
\hat\mu_{{m},0}(\t,z):=\mu_{{m},0}(\t,z)+ {i\over 2\sqrt{2m}}\sum_{r\xmod 2m} \th_{m,r}(\t,z) 	  \int_{-\bar\tau}^{i\inf} {\rm d}\t' {\theta^1_{m,r}(\t')\over\sqrt{-i(\t+\t')}}
\ee
where the sum is over $r\in \ZZ+m$.
Then $\hat\mu_{m,0}$ transforms as a Jacobi form of weight $1$ and index $m$. That is, $\hat\mu_{m,0}$ belongs to $\E_m$ (cf. \S\ref{sec:mjf:ell}) and is invariant for the weight $1$ modular action (\ref{eqn:modularweightkaction}) of $\SL_2(\ZZ)$ on $\E_m$. 

From (\ref{eqn:mucomp}) we see that if $\psi$ as in Proposition \ref{prop:Dswt1} has poles only at lattice points (i.e. $S(\psi)=S_1=\ZZ^2$), then the shadow of $h_r$ 
is proportional to the theta series $\theta^1_{m,r}(\tau)$ 
of weight $\frac32$, and the shadow of $\psi^F$ 
is proportional to the skew-holomorphic Jacobi form 
$\sum_{r\xmod 2m}\overline{\theta_{m,r}^1(\tau)}\theta_{m,r}(\tau,z)$
of weight $2$ and index $m$.

To describe the shadows of forms with poles at general torsion points we first note that for $m\in \frac12\ZZ$ and $n$ a divisor of $\tilde{m}$ we have
\be\label{eqn:thetamrWmn}
 (\theta_{m,r}\lvert {\cal W}_m{(n)})(\t,z) = \sum_{r'\xmod 2m } (\Omega_m(n)_{r,r'}) \theta_{m,r'}(\t,z)  
\ee
where $\W_m(n)$ is as in (\ref{eqn:Wmn}), and $\O_m(n)=(\O_m(n)_{r,r'})$ is the $2m\times 2m$ {\em Omega matrix} defined by setting
\begin{align}\label{def:OmegaMatrices}
\Omega_{m}(n)_{r,r'} := \begin{cases} \d^{[2n]}_{r+r',0} \d^{[{2m\over n}]}_{r-r',0} &\text{ if $m\in \ZZ$,} \\ 
\d^{[n]}_{r+r',0} \d^{[{2m\over n}]}_{r-r',0} 
&\text{ if $m\in \ZZ+{1\over 2}$.}
\end{cases}
\end{align}
In (\ref{def:OmegaMatrices}) the indices $r$ and $r'$ range over $\ZZ+m$ modulo translation by $2m$, and we use the notation 
\be
\d^{[N]}_{a,b} := \begin{cases} 1 & \text{ if $a=b\xmod{N}$,} \\ 0 & \text{ otherwise,}\end{cases}
\ee
for $a,b,N\in \ZZ$. For convenience let us also define the (non-holomorphic) {\em Eichler integral}
\begin{gather}
\theta_{m,r}^{1,*}(\tau):=\int_{-\bar\tau}^{i\inf} {\rm d}\t' {\theta^1_{m,r}(\t')\over\sqrt{-i(\t+\t')}}
\end{gather}
for $m\in \frac12\ZZ$ and $r\in \ZZ+m$.

\begin{thm}\label{non-thm}
Let $m\in\frac12\ZZ$ and let $\psi\in J^\mst_{1,m}$ be a weak meromorphic Jacobi form. 
Then 
\be\label{eqn_polar_AL}
\psi^P = \sum_{n|\tilde m} c_n  {\cal A}^{(0,0)}_{m}\lvert {\cal W}_m(n)
\ee
for some $c_n \in \CC$. 
Moreover, 
the completion
\be
\hat \psi^P(\tau,z) : = \psi^P(\tau,z)  + {i\over 2\sqrt{2m}}\sum_{r\xmod 2m} \th_{m,r}(\t,z) 	 \sum_{n|\tilde m} c_n \sum_{r'\xmod 2m}\Omega_{m}(n)_{r,r'} 
\theta_{m,r'}^{1,*}(\tau)
\ee
transforms as a weight 1 index $m$ Jacobi form. 
\end{thm}

\begin{proof}
First note that if $\phi\in \E_m$ is invariant for the weight $k$ index $m$ modular action (\ref{eqn:modularweightkaction}) of $\SL_2(\ZZ)$ then so is $\phi|\W_m(n)$ for $n$ a divisor of $\tilde{m}$. So the second statement follows from the first, together with the fact that $\hat\mu_{m,0}$ (cf. (\ref{eqn:mucomp})) transforms as a weight $1$ index $m$ Jacobi form.

For the first statement we first use (\ref{ellip_trans_S}) and (\ref{modular_trans_S}) to check that 
\be\label{D_sign}
D_{(-{a\over n},-{b\over n})} =\ex\left(m(a+1)(b+1)\right) D_{({1\over n},{1\over n})}
\ee
when $\gcd(a,b,n)=1$. Taking this together with Proposition \ref{prop:Dswt1} we obtain that
\be
\begin{split}
\psi^P(\t,z) :=&\sum_{s\in S(\psi)/\ZZ^2} D_s {\cal A}^{s}_{m}(\t,z) \\
=&    \sum_{n|\tilde m} {c_n\over n} \sum_{a,b=0}^{n-1}\ex(m(ab+a+b ))   {\cal A}^{(-\frac an,-\frac bn)}_{m}(\t,z) 
\end{split}
\ee
for some $c_n \in \CC$. Next we check that 
\begin{align}
 {\cal A}^{(-{a\over n},-{b\over n})}_{m}(\t,z) &= 
\ex\left(m\left(\tfrac{a^2}{n^2} \t + 2 \tfrac{a}{n}z +\tfrac{ab}{n^2}\right)\right)  {\cal A}^{(0,0)}\left(\t,z+\tfrac{a}{n}\t+\tfrac{b}{n}\right)
\end{align}
for $n$ a divisor of $\tilde m$. Then the first claim (\ref{eqn_polar_AL}) follows from 
the definition (\ref{eqn:Wmn}) of $\W_m(n)$.
\end{proof} 

From Theorem \ref{non-thm} we see that if $\psi$ is as in Proposition \ref{prop:Dswt1} then there are $c_n\in \CC$ for $n|\tilde{m}$ such that the shadow of $h_r$ 
is proportional to 
\begin{gather}
\sum_{n|\tilde{m}}\sum_{r'\xmod 2m} \overline{c_n}\Omega_m(n)_{r,r'}
\theta_{m,r}^1(\tau),
\end{gather}
and the shadow of $\psi^F$ 
is proportional to the skew-holomorphic Jacobi form 
\begin{gather}
\sum_{r\xmod 2m}\sum_{n|\tilde{m}}\sum_{r'\xmod 2m} c_n\Omega_m(n)_{r,r'}
\overline{\theta_{m,r'}^1(\tau)}\theta_{m,r}(\tau,z).
\end{gather}

\subsection{Umbral Meromorphic Forms}\label{sec:mjf:umbral}

Here we specialise the preceding analysis to the specific meromorphic Jacobi forms that will appear in \S\ref{sec:ummod}.

To begin we note that (\ref{eqn:mucomp}) and the identity $\th^1_{m,r} = {1\over 2} \ex({-\tfrac{r}{2}}) (\th^1_{4m,2r}+ \th^1_{4m,4m-2r})$ 
suggests a relation between half-integral meromorphic Jacobi forms and the umbral moonshine functions of pure D-type. 
To make this concrete,
define $M':=\frac12M+1$ for $M$ an even positive integer. 
Then for each 
positive $M$ such that $M=2\xmod 4$ and $M'$ is a divisor of $24$ there is a {\em pure D-type} Niemeier root system $X=D_{M'}^{\oplus d'}$ where $M'd'=24$. (In case $M=M'=2$ we should interpret $D_2$ as $A_1^{\oplus 2}$.) Given such an $M$ let $\ell$ be the symbol $M+\frac{M}{2}$. 
For instance, we have $\ell=6+3$ for $M=6$ and $M'=4$, corresponding to the Niemeier root system $D_4^{\oplus 6}$. 
We refer to $\ell$ as a lambency following \cite{UM,MUM}. 
Umbral moonshine \cite{MUM} attaches a $2M$-vector-valued mock modular form $H^{(\ell)}_g(\tau)=(H^{(\ell)}_{g,r}(\tau))_{r\xmod 2M}$ to each $g\in G^{(\ell)}$, where $G^{(\ell)}:=\Aut(N^X)/\Inn(N^X)$, for $N^X$ the Niemeier lattice with root system $X=D_{M'}^{\oplus d'}$, and $\Inn(N^X)$ the subgroup of $\Aut(N^X)$ generated by reflections in roots. 

For the particular lambencies $\ell=M+\frac{M}2$ under consideration, the mock modular property of $H^{(\ell)}_g$ can be formulated 
in the following way.
First note that $G^{(\ell)}\simeq S_d$ according to \S2.4 of \cite{MUM}. Denote the permutation character of $G^{(\ell)}\simeq S_d$ by $\chi^{(\ell)}_g$. 
(This is $\chi^{X_D}_g$ in the notation of \S B.2 of \cite{MUM}.)
Then 
for every conjugacy class $[g]$ of $G^{(\ell)}$,
\begin{align} \label{eqn:finite_polar}
\psi^{(\frac{M}4)}_{g}(\tau,z) :=
-2\chi^{(\ell)}_g {\rm Av}_{\frac{M}{4}}\left( 1\over y^{\frac12}-y^{-\frac12}\right)+
	\sum_{s =-\frac{M}{4}+1, -\frac{M}{4}+2, \dots, \frac{M}{4}}
	\ex(-\tfrac{s}2)H^{(\ell)}_{g,2s}(\tau)\theta_{\frac M4,s}(\tau,z)
\end{align}
is a meromorphic Jacobi form of weight $1$ and (half-integral)  index $\frac{M}4$, with some level, and the above decomposition corresponds precisely to the splitting of this meromorphic Jacobi form into its polar and finite parts, cf. Theorem \ref{thm:mjf:polarfinitedec}. 

Returning to the theta function identity mentioned above (take $m=\frac M4$), note that the discussion in \S\S4,5 of \cite{MUM} attaches a meromorphic Jacobi form $\psi^{(M+\frac M2)}_g$ of weight $1$ and index $M$ to each $g\in G^{(M+\frac M2)}$. While the polar parts of $\psi^{(\frac M4)}$ and $\psi^{(M+\frac M2)}$ do not seem to be related in a simple way, we do have $\psi^{(\frac{M}4),F}_{g}(\tau,2z) = \psi^{(M+{M\over 2}),F}_g(\tau,z)$ for the finite parts. The half-integral index forms $\psi^{(\frac M4)}$ seem to be better suited to explicit realisation in vertex algebraic terms.

Indeed, in \S\ref{sec:ummod} we will recover series expansions of the functions $\psi^{(\frac M4)}_g$ for $\frac{M}{4}\in\left\{\frac52,\frac72,\frac{11}2,\frac{23}2\right\}$ and $g\in G^{(M+\frac M2)}$ 
as traces on twisted modules $W^{(\frac M4)}_\tw$ for explicitly constructed vertex operator superalgebras $W^{(\frac M4)}$.

\section{Umbral Moonshine Modules}\label{sec:ummod}

In this section we present our main constructions, which are similar to those of \cite{umvan4}. We adopt all the notational conventions and terminology of \S2 of loc. cit. in what follows. 

\subsection{Lambency $10+5$}\label{sec:ummod:10+5}

Let $\gt{e}$ be a $2$-dimensional complex vector space equipped with a non-degenerate symmetric bilinear form, let $\gt{a}$ be a $4$-dimensional complex vector space equipped with a non-degenerate symmetric bilinear form, and let $\gt{b}$ be a $6$-dimensional complex vector space equipped with a non-degenerate anti-symmetric bilinear form. Choose polarisations $\gt{e}=\gt{e}^+\oplus\gt{e}^-$, $\gt{a}=\gt{a}^+\oplus\gt{a}^-$ and $\gt{b}=\gt{b}^+\oplus\gt{b}^-$, and let $\{e^\pm\}$, $\{a^\pm_i\}$ and $\{b_i^\pm\}$ be bases for $\gt{e}^\pm$, $\gt{a}^\pm$ and $\gt{b}^\pm$, respectively, such that $\lab e^-,e^+\rab=1$ and $\lab a_i^-,a_j^+\rab=\llab b_i^-,b_j^+\rab=\delta_{i,j}$.
Define a vertex operator superalgebra and a canonically twisted module for it by setting
\begin{gather}\label{eqn:ummod:4-WWtw}
\begin{split}
	W^{(\frac52)}&:=A(\gt{e})\otimes A(\gt{a})\,\otimes\uA(\gt{b}), \\
	W^{(\frac52)}_\tw&:=A(\gt{e})_\tw\otimes A(\gt{a})_\tw\,\otimes \uA(\gt{b})_\tw,
\end{split}
\end{gather}
and by equipping $W^{(\frac52)}$ with the Virasoro element $\omega^{(\frac52)}:=\omega\otimes\vv\otimes\vv+\vv\otimes\omega\otimes\vv+\vv\otimes\vv\otimes\omega$. 
Set $\jmath_\gt{e}:=\jmath\otimes\vv\otimes\vv$ and
$\jmath^{(\frac52)}:=2\vv\otimes\jmath\otimes\vv+\vv\otimes\vv\otimes \jmath$.
Then the group $\GL(\gt{e}^+)\otimes \GL(\gt{a}^+)\otimes\GL(\gt{b}^+)$ acts naturally on $W^{(\frac52)}$ and $W^{(\frac52)}_\tw$, respecting the vertex operator superalgebra module structures and preserving the bigradings defined by the zero modes of $\omega^{(\frac52)}$ and $\jmath^{(\frac52)}$.

\begin{table}[h]
\begin{center}
\begin{small}
\caption{\small Character table of ${G}^{(10+5)}\simeq S_4$}\label{tab:chars:irr:10+5}
\begin{tabular}{c|rrrrr}\toprule
$[g]$	&  1A& 2A&   3A&  2B&    4A\\
	\midrule
$[g^2]$ 	&1A&	1A&	3A&	1A&2A\\
	\midrule
${\chi}_{1}$&   $1$&   $1$&   $1$&  $1$&   $1$\\
${\chi}_{2}$&   $1$&   $1$&   $1$&   $-1$&   $-1$\\
${\chi}_{3}$&   $2$&   $2$&   $-1$&   $0$&   $0$\\
${\chi}_{4}$&   $3$&   $-1$&   $0$&   $1$&   $-1$\\
${\chi}_{5}$&   $3$&   $-1$&   $0$&   $-1$&   $1$\\
	\bottomrule
\end{tabular}
\end{small}
\end{center}
\end{table}

The group $G^{(10+5)}$ is isomorphic to $S_4$ according to \S2.4 of \cite{MUM}. Choose homomorphisms $\varrho:G^{(10+5)}\to \GL(\gt{a}^+)$ and $\uvarrho\;\;:G^{(10+5)}\to \GL(\gt{b}^+)$ such that the corresponding characters are $\chi_3$ and $\chi_4$, respectively, in the character table, Table \ref{tab:chars:irr:10+5}. 
Then the assignment $g\mapsto I\otimes \varrho(g)\,\otimes \uvarrho(g)$ defines faithful and compatible actions of $G^{(10+5)}$ on $W^{(\frac52)}$ and $W^{(\frac52)}_\tw$. Set $(-1)^F:=(-I)\otimes(-I)\otimes I$. Let $J_\gt{e}(0)$ be the coefficient of $z^{-1}$ in $Y_\tw(\jmath_\gt{e},z)$, let $J(0)$ be the coefficient of $z^{-1}$ in $Y_\tw(\jmath^{(\frac52)},z)$, and let $L(0)$ be the coefficient of $z^{-2}$ in $Y_\tw(\omega^{(\frac52)},z)$. For $g\in G^{(10+5)}$ we consider the formal series $\widetilde{\psi}^{(\frac52)}_g\in \CC[y][[y^{-1}]][[q]]$ defined by
\begin{gather}\label{eqn:ummod:7-trg}
	\widetilde{\psi}^{(\frac52)}_g:=-2\tr(gJ_\gt{e}(0)(-1)^Fy^{J(0)}q^{L(0)}|W^{(\frac52)}_\tw).
\end{gather}

\begin{thm}\label{thm:ummod:10+5-psig}
For $g\in G^{(10+5)}$ the series $\widetilde{\psi}^{(\frac52)}_g$ is the expansion of $\psi^{(\frac52)}_g$ in the domain $0<-\Im(z)<\Im(\tau)$.
\end{thm}
\begin{proof}
For $g\in G^{(10+5)}$ let $\{\lambda_1,\lambda_2\}$ be the eigenvalues for the action of $g$ on $\gt{a}^+$, and let $\{\ulambda_1,\ulambda_2,\ulambda_3\}$ be the eigenvalues for its action on $\gt{b}^+$. Then 
\begin{gather}	
	\label{eqn:ummod:10+5-psitildeprod}
\widetilde\psi^{(\frac52)}_g= 2y^{-\frac12} 
\prod_{n>0} 
	\frac {(1-q^n)^2\prod_{i=1}^2(1-\bar\lambda_i y^{-2}q^{n-1})(1-  \lambda_i y^{2}q^n)} 
	{\prod_{j=1}^3(1-\bar\ulambda_j  y^{-1} q^{n-1})(1- \ulambda_j \!y q^n)},
\end{gather}
where $(1-X)^{-1}$ is to be understood as a shorthand for $\sum_{k\geq 0}X^k$. This series converges in the domain $0<-\Im(z)<\Im(\tau)$ once we substitute $\ex(\tau)$ for $q$ and $\ex(z)$ for  $y$. So it remains to check that the right-hand side of (\ref{eqn:ummod:10+5-psitildeprod}) agrees with the meromorphic Jacobi form $\psi^{(\frac52)}_g$ when regarded as a function of $\tau$ and $z$. According to \S B.3.11 of \cite{umrec}, the $\psi^{(\frac52)}_g$ are given explicitly by
\begin{gather}
\begin{split}\label{eqn:ujf-psi5/2g}	
	\psi^{(\frac52)}_{1A}(\tau,z)&:=2i\eta(\tau)^3{\theta_1(\tau,2z)^2}{\theta_1(\tau,z)^{-3}},\\
	\psi^{(\frac52)}_{2A}(\tau,z)&:=-2i\eta(\tau)^3{\theta_1(\tau,2z)^2}{\theta_1(\tau,z)^{-1}\theta_2(\tau,z)^{-2}},\\
	\psi^{(\frac52)}_{3A}(\tau,z)&:=2i\eta(\tau)^3{\theta_1(3\tau,6z)}{\theta_1(\tau,2z)^{-1}\theta_1(3\tau,3z)^{-1}},\\
	\psi^{(\frac52)}_{2B}(\tau,z)&:=2i\eta(\tau)^3{\theta_1(\tau,2z)\theta_2(\tau,2z)}{\theta_1(\tau,z)^{-2}\theta_2(\tau,z)^{-1}},\\
	\psi^{(\frac52)}_{4A}(\tau,z)&:=-2i\eta(\tau)\eta(2\tau){\theta_1(\tau,2z)\theta_2(\tau,2z)}{\theta_2(2\tau,2z)^{-1}},
\end{split}
\end{gather}
where 
\begin{gather}
\begin{split}\label{eqn:ummod:10+5-jactheta}
	\eta(\tau)&:=q^{\frac1{24}}\prod_{n>0}(1-q^n),\\
	\theta_1(\tau,z)&:=-iq^{\frac18}y^{\frac12}\prod_{n>0}(1-y^{-1}q^{n-1})(1-yq^n)(1-q^n),\\
	\theta_2(\tau,z)&:=q^{\frac18}y^{\frac12}\prod_{n>0}(1+y^{-1}q^{n-1})(1+yq^n)(1-q^n).
\end{split}
\end{gather}
By applying these product formula definitions 
of $\eta$, $\theta_1$ and $\theta_2$ to the formulas in (\ref{eqn:ujf-psi5/2g}) we obtain product formulas for the $\psi^{(\frac52)}_g$. For each $g$ we find agreement with the product formula (\ref{eqn:ummod:10+5-psitildeprod}) for $\widetilde{\psi}^{(\frac52)}_g$ obtained by substituting the corresponding values for $\lambda_i$ and $\ulambda_j$, as are given in Table \ref{tab:ummod:10+5-evals}. This completes the proof.
\end{proof}

\begin{table}[ht]
\begin{center}
\begin{small}
\caption{\small Eigenvalues for $\ell=10+5$}
\begin{tabular}{ c | c  c c} \toprule\label{tab:ummod:10+5-evals}
    $[g]$ & $\{\lambda_i\}$ & $\{\ulambda_j\}$ \\ \midrule
    1A & $\{1,1\}$ & $\{1,1,1\}$\\ 
    2A & $\{1,1\}$ & $\{1,-1,-1\}$\\ 
    3A & $\{\omega,\omega^2\}$ & $\{1,\omega,\omega^2\}$\\ 
    2B & $\{1,-1\}$& $\{1,1,-1\}$\\
    4A & $\{1,-1\}$&$\{-1,\ii,-\ii\}$\\
    \bottomrule   
\end{tabular}
\end{small}
\end{center}
\end{table}

\subsection{Lambency $14+7$}\label{sec:ummod:14+7}

Let $\gt{e}$ and $\gt{a}$ be $2$-dimensional complex vector spaces equipped with non-degenerate symmetric bilinear forms, and let $\gt{b}$ be a $4$-dimensional complex vector space equipped with a non-degenerate anti-symmetric bilinear form. Choose polarisations $\gt{e}=\gt{e}^+\oplus\gt{e}^-$, $\gt{a}=\gt{a}^+\oplus\gt{a}^-$ and $\gt{b}=\gt{b}^+\oplus\gt{b}^-$, and let $\{e^\pm\}$, $\{a^\pm\}$ and $\{b_i^\pm\}$ be bases for $\gt{e}^\pm$, $\gt{a}^\pm$ and $\gt{b}^\pm$, respectively, such that $\lab e^-,e^+\rab=\lab a^-,a^+\rab=1$ and $\llab b_i^-,b_j^+\rab=\delta_{i,j}$.
Define a vertex operator superalgebra and a canonically twisted module for it by setting
\begin{gather}\label{eqn:ummod:7/2-WWtw}
\begin{split}
	W^{(\frac72)}&:=A(\gt{e})\otimes A(\gt{a})\,\otimes\uA(\gt{b}), \\
	W^{(\frac72)}_\tw&:=A(\gt{e})_\tw\otimes A(\gt{a})_\tw\,\otimes \uA(\gt{b})_\tw,
\end{split}
\end{gather}
and by equipping $W^{(\frac72)}$ with the usual tensor product Virasoro element, which we denote $\omega^{(\frac72)}$. 
Set $\jmath_\gt{e}:=\jmath\otimes\vv\otimes\vv$ and
$\jmath^{(\frac72)}:=3\vv\otimes\jmath\otimes\vv+\vv\otimes\vv\otimes \jmath$.
The group $\GL(\gt{e}^+)\otimes \GL(\gt{a}^+)\otimes\GL(\gt{b}^+)$ acts naturally on $W^{(\frac72)}$ and $W^{(\frac72)}_\tw$, respecting the vertex operator superalgebra module structures and preserving the bigradings defined by the Virasoro element $\omega^{(\frac72)}$ and the zero mode of $\jmath^{(\frac72)}$.

\begin{table}[h]
\begin{center}
\begin{small}
\caption{\small Character table of ${G}^{(14+7)}\simeq S_3$}\label{tab:chars:irr:14+7}
\begin{tabular}{c|rrr}\toprule
$[g]$&   1A&   2A&   3A\\
	\midrule
$\chi_1$&  $1$&   $1$&   $1$\\
$\chi_2$&  $1$&   $-1$&  $1$\\
$\chi_3$&  $2$&   $0$&   $-1$\\
\bottomrule
\end{tabular}
\end{small}
\end{center}
\end{table}

The group $G^{(14+7)}$ is isomorphic to $S_3$ according to \S2.4 of \cite{MUM}. Choose homomorphisms $\varrho:G^{(14+7)}\to \GL(\gt{a}^+)$ and $\uvarrho\;\;:G^{(14+7)}\to \GL(\gt{b}^+)$ such that the corresponding characters are the sign character and the unique irreducible character of dimension $2$, respectively (i.e. $\chi_2$ and $\chi_3$ in Table \ref{tab:chars:irr:14+7}).
Then the assignment $g\mapsto I\otimes \varrho(g)\,\otimes \uvarrho(g)$ defines 
actions of $G^{(14+7)}$ on $W^{(\frac72)}$ and $W^{(\frac72)}_\tw$. Set $(-1)^F:=(-I)\otimes(-I)\otimes I$, and let $J_\gt{e}(0)$ denote the coefficient of $z^{-1}$ in $Y_\tw(\jmath_\gt{e},z)$. Let $J(0)$ be the coefficient of $z^{-1}$ in $Y_\tw(\jmath^{(\frac72)},z)$, and let $L(0)$ be the coefficient of $z^{-2}$ in $Y_\tw(\omega^{(\frac72)},z)$. For $g\in G^{(14+7)}$ consider the formal series $\widetilde{\psi}^{(\frac72)}_g\in \CC[y][[y^{-1}]][[q]]$ defined by
\begin{gather}\label{eqn:ummod:7/2-trg}
	\widetilde{\psi}^{(\frac72)}_g:=-2\tr(gJ_\gt{e}(0)(-1)^Fy^{J(0)}q^{L(0)}|W^{(\frac72)}_\tw).
\end{gather}

\begin{thm}\label{thm:ummod:7/2-psig}
For $g\in G^{(14+7)}$ the series $\widetilde{\psi}^{(\frac72)}_g$ is the expansion of $\psi^{(\frac72)}_g$ in the domain $0<-\Im(z)<\Im(\tau)$.
\end{thm}

\begin{proof}
Let $g\in G^{(14+7)}$. Then $g$ acts as a multiplication by a scalar, $\lambda$ say, on $\gt{a}^+$, and we may write $\{\ulambda_1,\ulambda_2\}$ for the eigenvalues for its action on $\gt{b}^+$. As in \S\ref{sec:ummod:10+5} we have
\begin{gather}	
	\label{eqn:ummod:7/2-psitildeprod}
\widetilde\psi^{(\frac72)}_g= 2y^{-\frac12} 
\prod_{n>0} 
	\frac {(1-q^n)^2(1-\bar\lambda y^{-3}q^{n-1})(1-  \lambda y^{3}q^n)} 
	{\prod_{j=1}^{2}(1-\bar\ulambda_j  y^{-1} q^{n-1})(1- \ulambda_j \!y q^n)},
\end{gather}
and this series converges in the domain $0<-\Im(z)<\Im(\tau)$ upon substitution of $\ex(\tau)$ for $q$ and $\ex(z)$ for $y$. To check that the right-hand side of (\ref{eqn:ummod:7/2-psitildeprod}) agrees with the meromorphic Jacobi form $\psi^{(\frac72)}_g$ when viewed as a function of $\tau$ and $z$ we again perform a case by case check, using the
explicit expressions
\begin{gather}
\begin{split}\label{eqn:ujf-psi7/2g}		
	\psi^{(\frac72)}_{1A}(\tau,z)&:=2i\eta(\tau)^3{\theta_1(\tau,3z)}{\theta_1(\tau,z)^{-2}},\\
	\psi^{(\frac72)}_{2A}(\tau,z)&:=2i\eta(\tau)^3{\theta_2(\tau,3z)}{\theta_1(\tau,z)^{-1}\theta_2(\tau,z)^{-1}},\\
	\psi^{(\frac72)}_{3A}(\tau,z)&:=-2i\eta(3\tau)\th_1(\tau,z)\th_1(\tau,3z)\th_1(3\tau,3z)^{-1},
\end{split}
\end{gather}
reproduced here from \S B.3.15 of \cite{umrec}, and the values of $\lambda$ and $\ulambda_j$ in Table \ref{tab:ummod:14+7-evals}.
\end{proof}

\begin{table}[ht]
\begin{small}
\begin{center}
\caption{\small Eigenvalues for $\ell=14+7$}
\begin{tabular}{ c | c  c c} \toprule\label{tab:ummod:14+7-evals}
    $[g]$ & $\lambda$ & $\{\ulambda_j\}$ \\ \midrule
    1A & $1$ & $\{1,1\}$\\ 
    2A & $-1$ & $\{1,-1\}$\\ 
    3A & $1$ & $\{\omega,\omega^2\}$\\ 
    \bottomrule   
\end{tabular}
\end{center}
\end{small}
\end{table}

\subsection{Lambency $22+11$}\label{sec:ummod:22+11}

Let $\gt{e}$ and $\gt{a}$ be $2$-dimensional complex vector spaces equipped with non-degenerate symmetric bilinear forms, and let 
$\gt{b}$ and $\gt{b}'$ be $2$-dimensional complex vector spaces equipped with non-degenerate anti-symmetric bilinear forms. 
Fix polarisations $\gt{e}=\gt{e}^+\oplus\gt{e}^-$, $\gt{a}=\gt{a}^+\oplus\gt{a}^-$, $\gt{b}=\gt{b}^+\oplus\gt{b}^-$ and $\gt{b}'={\gt{b}'}^{+}\oplus {\gt{b}'}^-$, and let $\{e^\pm\}$, $\{a^\pm\}$, $\{b^\pm\}$ and $\{{b'}^{\pm}\}$ be bases for $\gt{e}^\pm$, $\gt{a}^\pm$, $\gt{b}^\pm$ and ${\gt{b}'}^\pm$, respectively, such that $\lab e^-,e^+\rab=\lab a^-,a^+\rab=\llab b^-,b^+\rab=\llab {b'}^-,{b'}^+\rab=1$.

Define a super vertex operator algebra $W^{(\frac{11}2)}$, and a canonically twisted $W^{(\frac{11}2)}$-module $W^{(\frac{11}2)}_\tw$ by setting
\begin{gather}\label{eqn:ummod:11/2-WWtw}
\begin{split}
	W^{(\frac{11}2)}&:=A(\gt{e})\otimes A(\gt{a})\,\otimes\uA(\gt{b})\otimes \uA(\gt{b}'), \\
	W^{(\frac{11}2)}_\tw&:=A(\gt{e})_\tw\otimes A(\gt{a})_\tw\,\otimes \uA(\gt{b})_\tw\otimes \uA(\gt{b}')_\tw.
\end{split}
\end{gather}
Equip $W^{(\frac{11}2)}$ with the usual tensor product Virasoro element, $\omega^{(\frac{11}2)}$, 
set $\jmath_\gt{e}:=\jmath\otimes\vv\otimes\vv\otimes\vv$, and set 
\begin{gather}
\jmath^{(\frac{11}2)}:=4\vv\otimes\jmath\otimes\vv\otimes\vv+\vv\otimes\vv\otimes \jmath\otimes\vv+2\vv\otimes\vv\otimes\vv\otimes\jmath.
\end{gather}
Then $\GL(\gt{e}^+)\otimes \GL(\gt{a}^+)\otimes\GL(\gt{b}^+)\otimes \GL({\gt{b}'}^+)$ acts naturally on $W^{(\frac{11}2)}$ and $W^{(\frac{11}2)}_\tw$, respecting the super vertex operator algebra module structures and preserving the bigradings. 

The 
umbral group $G^{(22+11)}$ is cyclic of order $2$ according to \S2.4 of \cite{MUM}. Define an action of $G^{(22+11)}$ on $W^{(\frac{11}2)}$ and $W^{(\frac{11}2)}_\tw$ by mapping the non-trivial element to $I\otimes I\otimes (-I)\otimes (-I)$. 
Similar to \S\S\ref{sec:ummod:10+5},\ref{sec:ummod:14+7} we set 
$(-1)^F:=(-I)\otimes(-I)\otimes I\otimes I$, 
let $J_\gt{e}(0)$ denote the coefficient of $z^{-1}$ in $Y_\tw(\jmath_\gt{e},z)$, let $J(0)$ be the coefficient of $z^{-1}$ in $Y_\tw(\jmath^{(\frac{11}2)},z)$, and let $L(0)$ be the coefficient of $z^{-2}$ in $Y_\tw(\omega^{(\frac{11}2)},z)$. Then to $g\in G^{(\frac{11}2)}$ we assign the formal series 
\begin{gather}\label{eqn:ummod:11/2-trg}
	\widetilde{\psi}^{(\frac{11}2)}_g:=-2\tr(gJ_\gt{e}(0)(-1)^Fy^{J(0)}q^{L(0)}|W^{(\frac{11}2)}_\tw).
\end{gather}

\begin{thm}\label{thm:ummod:11/2-psig}
For $g\in G^{(22+11)}$ the series $\widetilde{\psi}^{(\frac{11}2)}_g$ is the expansion of $\psi^{(\frac{11}2)}_g$ in the domain $0<-\Im(z)<\Im(\tau)$.
\end{thm}

\begin{proof}
We have
\begin{gather}	
	\label{eqn:ummod:11/2-psitildeprod}
\widetilde\psi^{(\frac{11}2)}_g= 2y^{-\frac12} 
\prod_{n>0} 
	\frac {(1-q^n)^2(1- y^{-4}q^{n-1})(1-  y^{4}q^n)} 
	{(1\mp y^{-1} q^{n-1})(1\mp y q^n)(1\mp y^{-2} q^{n-1})(1\mp y^2 q^n)},
\end{gather}
where the signs in the denominator are minus for $g=e$, and plus for $g$ the non-trivial element of $G^{(22+11)}$. In the former case the right hand side of (\ref{eqn:ummod:11/2-psitildeprod}) is the expansion of
\begin{gather}\label{ummod:11/2-psi1A}
2i\eta(\tau)^3\theta_1(\tau,4z)\theta_1(\tau,z)^{-1}\theta_1(\tau,2z)^{-1}
\end{gather}
in the domain $0<-\Im(z)<\Im(\tau)$, and (\ref{ummod:11/2-psi1A}) is exactly the expression for $\psi^{(\frac{11}2)}_{1A}$ that appears in \S B.3.19 of \cite{umrec}.
In the latter case (\ref{eqn:ummod:11/2-psitildeprod}) gives the expansion of
\begin{gather}\label{ummod:11/2-psi2A}
-2i\eta(\tau)^3\theta_1(\tau,4z)\theta_2(\tau,z)^{-1}\theta_2(\tau,2z)^{-1},
\end{gather}
which is the expression for $\psi^{(\frac{11}2)}_{2A}$ that appears in \S B.3.19 of \cite{umrec}. The proof is complete.
\end{proof}

\subsection{Lambency $46+23$}\label{sec:ummod:46+23}

The vertex operator superalgebra we use to realise the $\psi^{(\frac{23}2)}_g$ for $\ell=46+23$ is exactly the same as for $\ell=22+11$, but has a different bigrading. That is, we set $W^{(\frac{23}2)}:=W^{(\frac{11}2)}$ and $W^{(\frac{23}2)}_\tw:=W^{(\frac{11}2)}_\tw$, take $\jmath_\gt{e}$ as in \S\ref{sec:ummod:22+11}, and also define $\omega^{(\frac{23}2)}:=\omega^{(\frac{11}2)}$, but set
\begin{gather}
\jmath^{(\frac{23}2)}:=6\vv\otimes\jmath\otimes\vv\otimes\vv+2\vv\otimes\vv\otimes \jmath\otimes\vv+3\vv\otimes\vv\otimes\vv\otimes\jmath.
\end{gather}
The group $G^{(46+23)}$ is in fact trivial according to \S2.4 of \cite{MUM}, so we only aim to realise a single function $\psi^{(\frac{23}2)}_{1A}$. 
We define
$(-1)^F$, $J_\gt{e}(0)$ and $L(0)$ exactly as in \S\ref{sec:ummod:22+11}, but let $J(0)$ be the coefficient of $z^{-1}$ in $Y_\tw(\jmath^{(\frac{23}2)},z)$. We consider the formal series 
\begin{gather}\label{eqn:ummod:23/2-trg}
	\widetilde{\psi}^{(\frac{23}2)}_e:=-2\tr(J_\gt{e}(0)(-1)^Fy^{J(0)}q^{L(0)}|W^{(\frac{23}2)}_\tw).
\end{gather}

\begin{thm}\label{thm:ummod:23/2-psig}
The series $\widetilde{\psi}^{(\frac{23}2)}_e$ is the expansion of $\psi^{(\frac{23}2)}_{1A}$ in the domain $0<-\Im(z)<\Im(\tau)$.
\end{thm}
\begin{proof}
From \S B.3.23 of \cite{umrec} we have $\psi^{(\frac{23}2)}_{1A}=2i\eta(\tau)^3\theta_1(\tau,6z)\theta_1(\tau,2z)^{-1}\theta_1(\tau,3z)^{-1}$, which can be written as an infinite product,
\begin{gather}	
	\label{eqn:ummod:23/2-psitildeprod}
\psi^{(\frac{23}2)}_{1A}= 2y^{-\frac12} 
\prod_{n>0} 
	\frac {(1-q^n)^2(1- y^{-6}q^{n-1})(1-  y^{6}q^n)} 
	{(1- y^{-2} q^{n-1})(1- y^2 q^n)(1- y^{-3} q^{n-1})(1- y^3 q^n)}.
\end{gather}
On the other hand $\widetilde\psi^{(\frac{23}2)}_e$ is the series we obtain by replacing $(1-X)^{-1}$ with $\sum_{k\geq 0}X^k$ in (\ref{eqn:ummod:23/2-psitildeprod}), and converges in the given domain. This proves the claim. 
\end{proof}

\section{Discussion}\label{sec:disc}
 
In this section we discuss a few features and a possible interpretation 
of the pure D-type umbral moonshine module constructions presented in \S\ref{sec:ummod}.

\vspace{15pt}
\noindent
{\em Relation to ${\cal N}=2$ Superconformal Algebra}
\vspace{5pt}

Note that all the meromorphic Jacobi forms
\(\psi^{({M\over 4})}_g\) discussed in \S\ref{sec:ummod} have the property that $\psi^{({M\over 4})}_g( \Psi_{1,-\frac{1}{2}})^{-1}$ is a weight 0 weak (holomorphic) Jacobi form for $\SL_2(\ZZ)$ of index ${M\over 4}+{1\over 2}$, where
\be
\Psi_{1,-{1\over2}}(\t,z) := i {\eta^3(\tau)\over \theta_1(\tau,z) }
\ee
(cf. (\ref{eqn:ummod:10+5-jactheta})). Then the splitting of meromorphic Jacobi forms into polar and finite parts as described by Theorem \ref{thm:mjf:polarfinitedec}
corresponds precisely to the decomposition of the weight 0 weak Jacobi form
$\psi^{({M\over 4})}_g( \Psi_{1,-\frac{1}{2}})^{-1}$
into characters of the ${\cal N}=2$ superconformal algebra. In particular, the universal factor $( \Psi_{1,-\frac{1}{2}})^{-1}$ multiplying the polar and finite parts give the contributions from the massless and massive representations (respectively) of the ${\cal N}=2$ superconformal algebra of central charge $c=6({M\over 4}+{1\over 2})$.
  We refer to \S7 of \cite{2014arXiv1406.5502C} for details. 

This is reminiscent of the situation in the case $X=A_1^{\oplus 24}$, $\ell=2$ of umbral moonshine (cf. \cite{MUM}), namely Mathieu moonshine for $M_{24}$. There, analogous to the other 22 cases of umbral moonshine, we can regard the weight $\frac12$ mock modular form $H^{(2)}_g$ for all $g\in M_{24}$ as arising from splitting a certain weight 1 index 2 meromorphic Jacobi form $\psi^{(2)}_g$ into its polar and finite parts (see \S3.4 of \cite{MUM}). Alternatively, we can regard $H^{(2)}_g$ as arising from decomposing a weight 0 index 1 weak Jacobi form $\psi^{(2)}_g ( \Psi_{1,1})^{-1}$  into characters of the ${\cal N}=4$ superconformal algebra at central charge $6$ (see \S4.3 of \cite{MUM}), where
\be
\Psi_{1,1}(\t,z) := i {\eta^3(\tau) \theta_1(\tau,2z)\over \theta_1(\tau,z)^2 }. 
\ee
Famously, for the identity class $g=e$, the weak Jacobi form $\psi^{(2)}_g ( \Psi_{1,1})^{-1}$ equals the elliptic genus of K3 surfaces, and its ${\cal N}=4$ decomposition is precisely the context in which this $M_{24}$ moonshine was initially discovered \cite{Eguchi2010}.

\vspace{15pt}
\noindent
{\em Relation to Pure A-Type  Umbral Moonshine Modules}
\vspace{5pt}

For the other cases of pure A-type, corresponding to the Niemeier root systems $A_2^{\oplus 12}$, $A_3^{\oplus 8}$, $A_4^{\oplus 6}$, $A_6^{\oplus 4}$, $A_8^{\oplus 3}$, $A_{12}^{\oplus 2}$ and $A_{24}$, there is for general group elements $g$ no weight 0 weak Jacobi form related to $\psi^{(\ell)}_g$ via a multiplication of $(\Psi_{1,1})^{-1}$ due to the presence of poles not just at lattice points $z\in \ZZ\tau+\ZZ$ but also at $2$-torsion points $z\in \ZZ\tau+\ZZ+{1\over 2}$. (Note the contrast with the $\psi^{(\frac M4)}_g$ in \S\ref{sec:ummod}, whose poles are restricted to lattice points for all $g$.)   However, 
for all the pure A-type cases  $A_{\ell-1}^{\oplus d}$ where $d={24\over \ell-1}$, 
the function $\psi^{(\ell)}_e ( \Psi_{1,1})^{-1}$ is a 
weight 0 weak Jacobi form  for $\SL_2(\ZZ)$ 
corresponding to the identity class. 
 
As a result, we see that the graded dimension of the pure A-type and pure D-type umbral moonshine modules, captured by weight 1 meromorphic Jacobi forms at integral and half-integral indices respectively, are both related to weak Jacobi forms of weight 0. 
In fact, the pairs of pure D-type and pure A-type cases of umbral moonshine listed in 
Table \ref{tab:ADmap}  give rise to the same weak Jacobi forms. Precisely, if $M$ is as in Table \ref{tab:ADmap} and $M':=\frac12M+1$ then $\psi^{(\frac{M}4)}_e$ is the weight 1 index $\frac M4$ meromorphic Jacobi form that represents the $D_{M'}^{d'}$ case of umbral moonshine (cf. \S\ref{sec:mjf:umbral}), where $M'd'=24$, and the corresponding pure A-type case of umbral moonshine is $A_{M''-1}^{d''}$, where $M''=(M')'=\frac14 M+\frac32$ and $(M''-1)d''=24$. The corresponding meromorphic Jacobi forms are related by
 \be
 \psi^{({M\over 4})}_e = {\theta_1(\tau,z)\over\theta_1(\tau,2z)} \psi^{(M'')}_e.
 \ee
In other words, in these four cases (and also for $M=6$, which is not considered in this work) the weight 0 weak Jacobi forms $\psi^{({M\over 4})}_e( \Psi_{1,-\frac{1}{2}})^{-1} =  \psi^{(M'')}_e ( \Psi_{1,1})^{-1}$  can be interpreted as encoding graded dimensions for either a pure D-type or a pure A-type moonshine module, the former constructed in the present paper and the latter constructed (in whole for $M\in \{22,26\}$ but in part for $M\in\{10,14\}$) in \cite{umvan4}.

 \begin{table}[h]
 \begin{center}
 \begin{small}
 \caption{\small The A--D Correspondence}\label{tab:ADmap}
\begin{tabular}{ccc}\toprule
$M$ & D-type & A-type \\\midrule
$10$&$D_6^{\oplus 4}$  & $A_3^{\oplus 8}$\\\midrule 
$14$&$D_8^{\oplus 3}$  & $A_4^{\oplus 6}$\\\midrule
$22$&$D_{12}^{\oplus 2}$  & $A_6^{\oplus 4}$\\\midrule
$46$&$D_{24}$  & $A_{12}^{\oplus 2}$\\\bottomrule
\end{tabular}
\end{small} 
\end{center}
\end{table}

\vspace{15pt}
\noindent
{\em Possible Interpretation of the Module}
\vspace{5pt}

Note that in each of the four pure D-type cases of  umbral moonshine discussed in this paper, the corresponding module construction in \S\ref{sec:ummod} can be interpreted as given by $d$ pairs of $bc$-$\beta\gamma$ systems, where $d=3$ for $\ell=10+5$ (cf. \S\ref{sec:ummod:10+5}), and $d=2$ in the remaining cases (cf. \S\S\ref{sec:ummod:14+7}-\ref{sec:ummod:46+23}). 
This is reminiscent of the construction of the chiral de Rham complex \cite{MR1704283}, 
which defines a sheaf of superconformal vertex operator superalgebras over a Calabi--Yau manifold of complex dimension $d$. In that construction, each local section of the chiral de Rham complex is given by  $d$ pairs of $bc$-$\beta\gamma$ systems, and 
has the structure of a so-called rank $d$ topological ${\cal N}=2$ superconformal algebra. In particular the
OPE of the stress-energy tensor with itself gives a vanishing central charge.
Note however that our choices of $U(1)$ charges for the different $bc$ and $\beta\gamma$ systems do not preserve this topological ${\cal N}=2$ structure. Nonetheless it would be interesting to explore whether this close connection to the chiral de Rham complex, present also in \cite{umvan4}, indeed reflects a hitherto unnoticed physical aspect to umbral moonshine. 

\section*{Acknowledgements}

We thank Andrew O'Desky for discussions on closely related topics. 
The work of M.C. was supported by ERC starting grant H2020 ERC StG \#640159. 
J.D. acknowledges support from the Simons Foundation (\#316779), and the U.S. National Science Foundation (DMS 1203162, DMS 1601306).

\addcontentsline{toc}{section}{References}

\providecommand{\href}[2]{#2}\begingroup\raggedright\endgroup


\begin{thebibliography}{10}

\bibitem{Eguchi2010}
T.~Eguchi, H.~Ooguri, and Y.~Tachikawa, ``{Notes on the K3 Surface and the
  Mathieu group $M_{24}$},'' {\em Exper.Math.} {\bf 20} (2011)  91--96,
\href{http://arxiv.org/abs/1004.0956}{{\tt arXiv:1004.0956 [hep-th]}}.

\bibitem{UM}
M.~C.~N. Cheng, J.~F.~R. Duncan, and J.~A. Harvey, ``{Umbral Moonshine},''
  \href{http://dx.doi.org/10.4310/CNTP.2014.v8.n2.a1}{{\em Commun. Number
  Theory Phys.} {\bf 8} (2014) no.~2, 101--242},
  \href{http://arxiv.org/abs/1204.2779}{{\tt arXiv:1204.2779 [math.RT]}}.
\url{http://dx.doi.org/10.4310/CNTP.2014.v8.n2.a1}.

\bibitem{MUM}
M.~C.~N. Cheng, J.~F.~R. Duncan, and J.~A. Harvey, ``{Umbral Moonshine and the
  Niemeier Lattices},'' {\em Research in the Mathematical Sciences} {\bf 1}
  (2014) no.~3, 1--81,
\href{http://arxiv.org/abs/1307.5793}{{\tt arXiv:1307.5793 [math.RT]}}.

\bibitem{omjt}
M.~C.~N. {Cheng} and J.~F.~R. {Duncan}, ``{Optimal Mock Jacobi Theta
  Functions},''{\em ArXiv e-prints} (May, 2016)  ,
  \href{http://arxiv.org/abs/1605.04480}{{\tt arXiv:1605.04480 [math.NT]}}.

\bibitem{MR3539377}
T.~Gannon, ``Much ado about {M}athieu,''
  \href{http://dx.doi.org/10.1016/j.aim.2016.06.014}{{\em Adv. Math.} {\bf 301}
  (2016)  322--358}. \url{http://dx.doi.org/10.1016/j.aim.2016.06.014}.

\bibitem{umrec}
J.~F.~R. {Duncan}, M.~J. {Griffin}, and K.~{Ono}, ``{Proof of the Umbral
  Moonshine Conjecture},''{\em Research in the Mathematical Sciences} {\bf 2}
  (Mar., 2015)  , \href{http://arxiv.org/abs/1503.01472}{{\tt arXiv:1503.01472
  [math.RT]}}.

\bibitem{MR554399}
J.~H. Conway and S.~P. Norton, ``Monstrous moonshine,''
  \href{http://dx.doi.org/10.1112/blms/11.3.308}{{\em Bull. London Math. Soc.}
  {\bf 11} (1979) no.~3, 308--339}.
  \url{http://dx.doi.org/10.1112/blms/11.3.308}.

\bibitem{MR1172696}
R.~E. Borcherds, ``Monstrous moonshine and monstrous {L}ie superalgebras,''
  \href{http://dx.doi.org/10.1007/BF01232032}{{\em Invent. Math.} {\bf 109}
  (1992) no.~2, 405--444}. \url{http://dx.doi.org/10.1007/BF01232032}.

\bibitem{FLMBerk}
I.~B. Frenkel, J.~Lepowsky, and A.~Meurman, ``A moonshine module for the
  {M}onster,'' in {\em Vertex operators in mathematics and physics (Berkeley,
  Calif., 1983)}, vol.~3 of {\em Math. Sci. Res. Inst. Publ.}, pp.~231--273.
\newblock Springer, New York, 1985.

\bibitem{FLMPNAS}
I.~B. Frenkel, J.~Lepowsky, and A.~Meurman, ``A natural representation of the
  {F}ischer-{G}riess {M}onster with the modular function {$J$} as character,''
  {\em Proc. Nat. Acad. Sci. U.S.A.} {\bf 81} (1984) no.~10, Phys. Sci.,
  3256--3260.

\bibitem{Bor_PNAS}
R.~Borcherds, ``Vertex algebras, {Kac}-{Moody} algebras, and the {Monster},''
  {\em Proceedings of the National Academy of Sciences, U.S.A.} {\bf 83} (1986)
  no.~10, 3068--3071.

\bibitem{FLM}
I.~B. Frenkel, J.~Lepowsky, and A.~Meurman, {\em Vertex operator algebras and
  the {M}onster}, vol.~134 of {\em Pure and Applied Mathematics}.
\newblock Academic Press Inc., Boston, MA, 1988.

\bibitem{MR3649360}
J.~Duncan and J.~Harvey, ``The umbral moonshine module for the unique
  unimodular {N}iemeier root system,''
  \href{http://dx.doi.org/10.2140/ant.2017.11.505}{{\em Algebra Number Theory}
  {\bf 11} (2017) no.~3, 505--535}.
  \url{http://dx.doi.org/10.2140/ant.2017.11.505}.

\bibitem{umvan4}
J.~F.~R. {Duncan} and A.~{O'Desky}, ``{Super Vertex Algebras, Meromorphic
  Jacobi Forms and Umbral Moonshine},''{\em ArXiv e-prints} (May, 2017)  ,
  \href{http://arxiv.org/abs/1705.09333}{{\tt arXiv:1705.09333 [math.RT]}}.

\bibitem{zwegers}
S.~Zwegers, {\em {Mock Theta Functions}}.
\newblock PhD thesis, Utrecht University,
  http://igitur-archive.library.uu.nl/dissertations/2003-0127-094324/inhoud.htm.,
  2002.

\bibitem{Dabholkar:2012nd}
A.~Dabholkar, S.~Murthy, and D.~Zagier, ``{Quantum Black Holes, Wall Crossing,
  and Mock Modular Forms},''
\href{http://arxiv.org/abs/1208.4074}{{\tt arXiv:1208.4074 [hep-th]}}.

\bibitem{eichler_zagier}
M.~Eichler and D.~Zagier, {\em {The theory of Jacobi forms}}.
\newblock Birkh{\"a}user, 1985.

\bibitem{ZWEGERS2010MultivariableAF}
S.~Zwegers, ``{Multivariable Appell Functions},''
\newblock 2010.

\bibitem{2014arXiv1406.5502C}
M.~C.~N. {Cheng}, X.~{Dong}, J.~F.~R. {Duncan}, S.~{Harrison}, S.~{Kachru}, and
  T.~{Wrase}, ``{Mock Modular Mathieu Moonshine Modules},'' {\em Research in
  the Mathematical Sciences} {\bf 2} (2015) no.~13, ,
  \href{http://arxiv.org/abs/1406.5502}{{\tt arXiv:1406.5502 [hep-th]}}.

\bibitem{MR1704283}
F.~Malikov, V.~Schechtman, and A.~Vaintrob, ``Chiral de {R}ham complex,''
  \href{http://dx.doi.org/10.1007/s002200050653}{{\em Comm. Math. Phys.} {\bf
  204} (1999) no.~2, 439--473}. \url{http://dx.doi.org/10.1007/s002200050653}.

\end{thebibliography}
\end{document}